\documentclass[10pt]{amsart}
\usepackage{amssymb}
\usepackage{epsfig}
\usepackage{url}
\usepackage{setspace}
\usepackage{color}
\theoremstyle{plain}

\newtheorem{thm}{Theorem}[section]
\newtheorem{cor}[thm]{Corollary}
\newtheorem{lem}[thm]{Lemma}

\newtheorem{rem}[thm]{Remark}
\newtheorem{ques}[thm]{Question}
\newtheorem{conj}[thm]{Conjecture}

\def\bbb{\mathbb}
\def\op{\operatorname}
\renewcommand{\phi}{\varphi}
\newcommand{\R}{\bbb{R}}
\newcommand{\N}{\bbb{N}}
\newcommand{\Z}{\bbb{Z}}
\newcommand{\Q}{\bbb{Q}}

\newcommand{\eps}{\varepsilon}

\begin{document}
\title[Strong arithmetic property of certain Stern polynomials]{Strong arithmetic property of certain Stern polynomials}
\author{Maciej Ulas}


\begin{abstract}
Let $B_{n}(t)$ be the $n$th Stern polynomial, i.e., the $n$th term of the sequence defined recursively as $B_{0}(t)=0, B_{1}(t)=1$ and $B_{2n}(t)=tB_{n}(t), B_{2n+1}(t)=B_{n}(t)+B_{n-1}(t)$ for $n\in\N$. It is well know that $i$th coefficient in the polynomial $B_{n}(t)$ counts the number of hyperbinary representations of $n-1$ containing exactly $i$ digits 1. In this note we investigate the existence of odd solutions of the congruence
\begin{equation*}
B_{n}(t)\equiv 1+rt\frac{t^{e(n)}-1}{t-1}\pmod{m},
\end{equation*}
where $m\in\N_{\geq 2}$ and $r\in\{0,\ldots,m-1\}$ are fixed and $e(n)=\op{deg}B_{n}(t)$. We prove that for $m=2$ and $r\in\{0,1\}$ and for $m=3$ and $r=0$, there are infinitely many odd numbers $n$ satisfying the above congruence. We also present results of some numerical computations.
\end{abstract}

\maketitle

\section{Introduction}\label{sec1}

Let $\N=\{0,1,\ldots, \}, \N_{+}=\{1,2,\ldots,\}$ and for $k\geq 2$ we put $\N_{\geq k}:=\{k,k+1,\ldots\}$.

Let $n\in\N$ and recall that a representation of the form
\begin{equation*}\label{hyperbin}
n=\sum_{i=0}^{k}\eps_{i}2^{i},
\end{equation*}
where $\eps_{i}\in\{0,1,2\}$ is called {\it a hyperbinary representation of $n$}. The sequence $(s_{n})_{n\in\N}$ counting all hyperbinary representations of $n$ is called {\it the Stern sequence} or the Stern diatomic sequence \cite{Ste}. One can check that the Stern sequence satisfies the recurrence relation(s)
\begin{equation*}
s_{0}=0, s_{1}=1,\quad s_{2n}=s_{n},\quad s_{2n+1}=s_{n}+s_{n+1}
\end{equation*}
for $n\in\N_{+}$. The first 32 elements of the Stern sequence are as follows
\begin{equation*}
0, 1, 1, 2, 1, 3, 2, 3, 1, 4, 3, 5, 2, 5, 3, 4, 1, 5, 4, 7, 3, 8, 5, 7, 2, 7, 5, 8, 3, 7, 4, 5, \ldots.
\end{equation*}
Among many interesting properties of the Stern sequences is the one given by Calkin and Wilf which says that the sequence $(s_{n}/s_{n+1})_{n\in\N_{+}}$ contains each positive rational number exactly once \cite{CalWil}. Many other results concerning $(s_{n})_{n\in\N}$ can be found in the classical paper of Lehmer \cite{Leh} and more recent paper of Urbiha \cite{Ur}.

The concept of the Stern sequence was generalized in many directions. The polynomial generalization was proposed by Klav\v{z}ar, Milutinovi\'{c} and Petr in \cite{KMP}. More precisely, they define the sequence $(B_{n}(t))_{n\in\N}$ of {\it the Stern polynomials} in the following way: $B_{0}(t)=0, B_{1}(t)=1$ and for $n\in\N_{+}$ we have
\begin{equation*}
B_{2n}(t)=tB_{n}(t),\quad B_{2n+1}(t)=B_{n}(t)+B_{n+1}(t).
\end{equation*}
We have $B_{n}(1)=s_{n}$ and one can easily check by induction on $n$ that $B_{n}(2)=n$. Moreover, there is a striking combinatorial property of $B_{n}(t)$ connecting its coefficients with the number of (certain) hyperbinary representations of $n-1$. More precisely, if we put $e(n)=\op{deg}B_{n}(t)$ and write
$$
B_{n}(t)=\sum_{i=0}^{e(n)}a(i,n-1)t^{i},
$$
then $a(i,n-1)$ is the number of hyperbinary representations of $n-1$ containing exactly $i$ digits 1. Let us note that a combinatorial interpretation of the polynomial $B_{n}(t)$ can be seen from the shape of the ordinary generating function
$$
\sum_{n=0}^{\infty}B_{n}(t)x^{n}=x\prod_{n=0}^{\infty}\left(1+tx^{2^{n}}+x^{2^{n+1}}\right).
$$

Let us also note that the sequence $(e(n))_{n\in\N_{+}}$ of degrees of the Stern polynomials has also interesting properties (some are noted in \cite{KMP} and many others are presented in \cite{Ul2}).

The paper of Klav\v{z}ar, Milutinovi\'{c} and Petr motivated a lot of research devoted to various properties of Stern polynomials (see for example \cite{Ul1, Ul2, Sch, Sch1, Sch2, Gaw, DT, DKT} and reference given therein). However, it seems that only the paper \cite{Sch2} deals with the partition properties of the coefficients of $B_{n}(t)$. More precisely, the computation of $a(e(n),n-1)$ in terms of the (unique) binary expansion of $n$ is presented \cite[Theorem 1]{Sch2}. This is an interesting result which gives value of the number of hyperbinary representations of $n-1$ with maximal possible number of 1's. As a corollary, Schinzel obtained a striking result which says that for any given $m\in\N_{\geq 2}$, the density of these $n\in\N_{+}$ such that $a(e(n),n-1)\equiv 0\pmod{m}$ is equal to 1 (and thus $\limsup_{n\rightarrow+\infty}a(e(n),n-1)=+\infty$). The cited result motivated our research. However, instead of trying to obtain congruences properties for the individual coefficient $a(i,n-1)$, we study the more tractable problem of finding the numbers $n$ for which {\it all} coefficients $a(i,n-1), 1\leq i\leq e(n)$, are congruent to some fixed value $r\pmod{m}$. This is equivalent to finding solutions of the polynomial congruence
\begin{equation}\label{maincong}
B_{n}(t)\equiv 1+r\sum_{i=1}^{e(n)}t^{i}\pmod{m},
\end{equation}
where $m\in\N_{\geq 2}, r\in\{0,\ldots, m-1\}$ are fixed. The existence of a solution of (\ref{maincong}) for given $m$ and $r$ has strong combinatorial interpretation. Indeed, let $n\in\N_{+}$ be odd and write
$$
B_{n}(t)=1+\sum_{i=1}^{e(n)}a(i,n-1)t^{i}.
$$
Then, if (\ref{maincong}) holds this means that
$$
a(1,n-1)\equiv a(2,n-1)\equiv\ldots\equiv a(e(n),n-1)\equiv r\pmod{m}.
$$
This is a strong property and a question arises whether for a given pair $(r, m)$ it is possible to find a solution of (\ref{maincong}). In fact, we are interested in finding, for a given pair $(r,m)$ an infinite family of solutions. However, we are interested in non-trivial families only. More precisely, by a non-trivial family of solutions of (\ref{maincong}) we understand a sequence $(a_{n})_{n\in\N}$ of positive odd integers, such that the set of coefficients of all polynomials $B_{a_{n}}(t), n\in\N$, is infinite. This assumption is quite strong but we believe that that it is reasonable due to the fact that for a given odd $k$, the set of coefficients of all polynomials $(B_{2^{n}-k}(t))_{n\in\N\geq \lfloor\log_{2} k\rfloor+1}$ is finite. Indeed, this is a consequence of the identity
\begin{equation*}\label{eq1}
B_{2^{n}-k}(t)=B_{k}(t)\frac{t^{n-m}-1}{t-1}-B_{l}(t)t^{n-m},
\end{equation*}
where $k=2^{m}+l$, obtained in a recent paper of Dilcher, Kidwai and Tomkins \cite{DKT}. In particular, the number of different coefficients of all polynomials $B_{2^{n}-k}(t)$ is bounded by $k+l$, and thus the family $(B_{2^{n}-k}(t))_{n\in\N\geq \lfloor\log_{2} k\rfloor+1}$ is trivial for all $k$. For given $m\geq 3$ and $r\in\{1,2\}$ one can also obtain trivial families of solutions of (\ref{maincong}). Indeed, if $r=1$, then for $n\in\N$ we have the identity
$$
B_{2^{n+1}-1}(t)=\sum_{i=0}^{n}t^{i},
$$
and the congruence (\ref{maincong}) is trivially satisfied. Similarly, if $r=2$ we have the identity
$$
B_{2^{n+2}-3}(t)=1+2\sum_{i=1}^{n}t^{i},
$$
and thus for any $m\in\N_{\geq 3}$ there are infinitely many solutions of (\ref{maincong}). The existence of examples of this kind is the reason why we are interested in non-trivial families only. Thus, in the sequel, by a family of solutions of (\ref{maincong}) we mean a non-trivial family.

After the above discussion let us describe the content of the paper in some detail. In Section \ref{sec2} we prove the existence of an infinite family of non-trivial sequences $(p_{k,n})_{n\in\N_{+}}, k\in\N_{\geq 2}$, of solutions of the congruence (\ref{maincong}) for $m=2$ and $r=0$. Moreover, we present a lower bound for the number of solutions. In case of $m=2$ and $r=1$ we prove the existence of a four non-trivial sequence of solutions. We also present results of our numerical computations.

In Section \ref{sec3} we deal with (\ref{maincong}) for $m=3$. Our main result states that there is a non-trivial sequence of solutions corresponding to $r=0$.

Finally, in Section \ref{sec4} we present some examples of solutions of (\ref{maincong}) for certain values $m\in\{4,\ldots, 10\}, r\in\{0,\ldots,m-1\}$ and formulate some additional questions and conjectures.


\section{The case of $m=2$}\label{sec2}

Before we prove the main result of this section let us recall the following useful identities obtained by Schinzel in \cite{Sch} (see also a note in  \cite[Lemma 2.1]{DT}).

\begin{lem}\label{lem1}
For all nonnegative integers $a, m$, and $r$ with $0\leq r\leq 2^{a}$ we have
\begin{align*}
B_{m2^{a}+r}(t)&=B_{2^{a}-r}(t)B_{m}(t)+B_{r}(t)B_{m+1}(t),\\
B_{m2^{a}-r}(t)&=B_{2^{a}-r}(t)B_{m}(t)+B_{r}(t)B_{m-1}(t).
\end{align*}
\end{lem}

Let us also note the following identities which will be used in the sequel.

\begin{lem}\label{lem2}
We have the following identities
\begin{align*}
B_{2^{n}-1}(t)&=\frac{t^{n}-1}{t-1}, n\in\N_{+}\\
B_{2^{n}-3}(t)&=t\frac{t^{n-2}-1}{t-1}+\frac{t^{n-1}-1}{t-1}, n\in\N_{\geq 2},\\
B_{2^{n}-5}(t)&=t\frac{t^{n-3}-1}{t-1}+(t+1)\frac{t^{n-2}-1}{t-1}, n\in\N_{\geq 3},\\
B_{2^{n}-9}(t)&=t\frac{t^{n-4}-1}{t-1}+(t^2+t+1)\frac{t^{n-3}-1}{t-1}, n\in\N_{\geq 4}.
\end{align*}
\end{lem}

Now we are ready to prove the result which says that for $m=2$ and fixed $r=0$ there is an infinite family of sequences $(p_{k,n})_{n\in\N_{+}}$ such that for each $k\in\N_{\geq 2}$ the corresponding sequence gives a non-trivial family of solutions of (\ref{maincong}). However, before we present the statement of the result we are talking about, it is better to explain how the family $(p_{k,n})_{n\in\N_{+}}$ was found.

For $m\in\N_{\geq 2}$ and $r\in\{0,\ldots, m-1\}$ let us put
$$
\Pi_{r,m}(x):=|\{n\leq x:\;\mbox{the congruence}\;(\ref{maincong})\;\mbox{holds for}\;n\}|.
$$

First, we found all odd solutions of (\ref{maincong}) satisfying the condition $n\leq 2^{26}$. In the table below we present the number $\Pi_{r,m}(2^{k})$ of solutions of (\ref{maincong}) for $(r,m)\in\{(0,2), (1,2)\}$ and $k\in\{15,\ldots, 26\}$. Moreover, we present the graph of the function $\Pi_{r,2}(x)$ (Figure 1) together with the graph of the quotient $\Pi_{0,2}(x)/\Pi_{1,2}(x)$ (Figure 2) in the range $0<x\leq 2^{20}$.

\begin{equation*}
\begin{array}{|c|llllllllllll|}
\hline
 k                & 15 & 16 & 17 & 18 & 19 & 20 & 21 & 22 & 23 & 24 & 25 & 26 \\
 \hline
 \Pi_{0,2}(2^{k}) & 97 & 136 & 185 & 253 & 339 & 453 & 609 & 819 & 1121 & 1519 & 2057 & 2841 \\
 \hline
 \Pi_{1,2}(2^{k}) & 82 & 115 & 146 & 182 & 217 & 258 & 311 & 371 & 441  & 522  & 629  & 783 \\
 \hline
\end{array}
\end{equation*}
\begin{center}
Table 1. The number $\Pi_{r,m}(2^{k})$ of solutions of (\ref{maincong}) for $(r,m)\in\{(0,2), (1,2)\}$ and $k\in\{15,\ldots, 26\}$.
\end{center}

\begin{figure}[htbp]\label{Pic1} 
       \centering
         \includegraphics[width=3.5in]{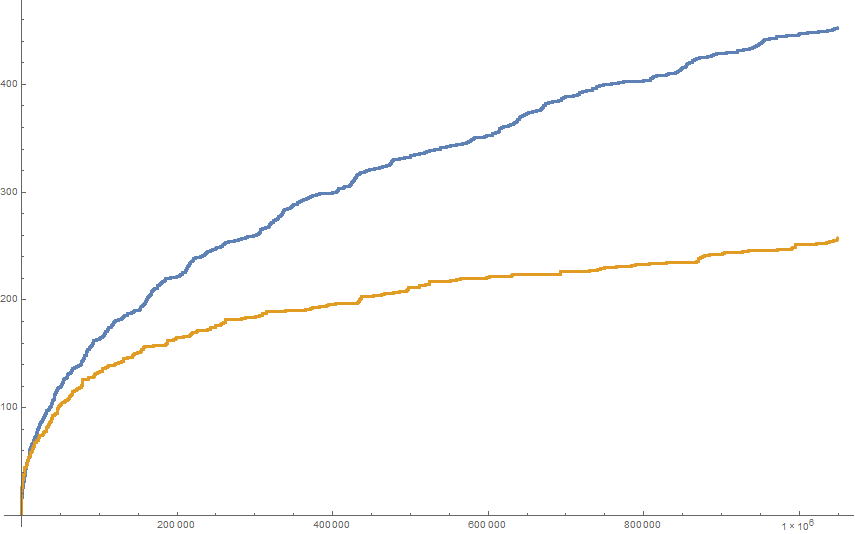}
        \caption{Plot of the function $\protect \Pi_{0,2}(x)$ (up) and $\protect \Pi_{1,2}(x)$ (down) for $\protect x\leq 2^{20}$}
       \label{fig:disc}
    \end{figure}

\begin{figure}[htbp]\label{Pic2} 
      \centering
         \includegraphics[width=3.5in]{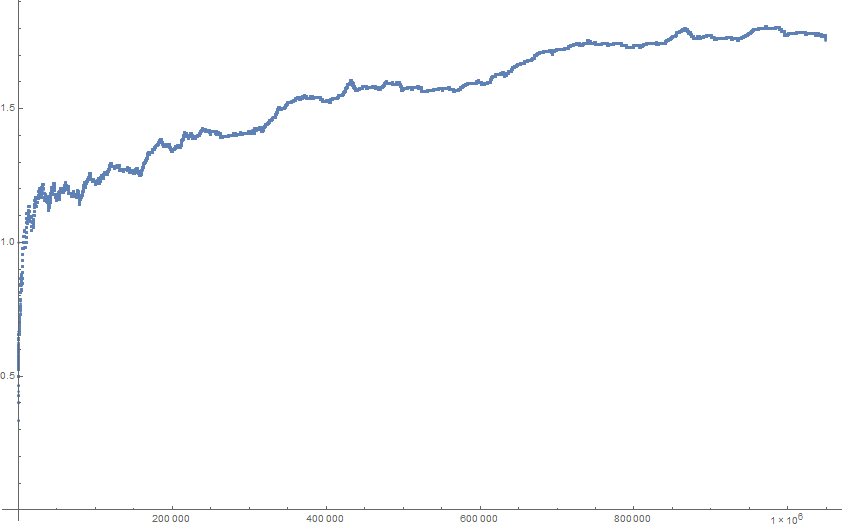}
        \caption{Plot of the function $\protect \Pi_{0,2}(x)/\Pi_{1,2}(x)$ for $\protect 5\leq x\leq 2^{20}$}
       \label{fig:disc}
    \end{figure}

Having a large table of solutions (containing also the solutions which are not interesting from our point of view) we proceed in the following way. We consider the set, say $A_{0,2}$, of solutions which are $\leq 4\cdot 10^4$. There are exactly 106 elements in the set $A_{0,2}$. We expected that there should be a family of solutions of the form $U_{n}=p2^{2n}+q2^{n}+u$ for certain values of $p, q, u\in\Q$. Our expectation follows from the observation of binary expansions of elements of the set $A_{0,2}$. In order to find these families, we compute the set $B_{0,2}$ containing all quadruplets $(a,b,c,d)$ such that $a, b, c, d\in A_{0,2}$ and $a<b<c<d$. The set $B_{0,2}$ has 4967690 elements. Finally, for any $v\in B_{0,2}$ we look for triples $(p, q, u)$ such that $v=(U_{0},U_{1},U_{2},U_{3})$. We thus work with four linear equations in three variables, and it is quite natural expectation that the existence of a solution in this case implies the existence of an {\it infinite} family of solutions of (\ref{maincong}).
In this way we obtain the set, say $C_{0,2}$, of 523 triplets $(p,q,u)$. However, not all triplets in the set $C_{0,2}$ give a solution we are interested in. For example, the quadruple $v=(5,29,253,1405)\in B_{0,2}$ gives the triple $(p, q, u)=(88/3,-64,119/3)$ and the corresponding sequence $U_{n}$. However, we have $U_{4}=6525$ with the corresponding polynomial $B_{U_{4}}(t)=(t+1)(6 t^7+10 t^6+12 t^5+13 t^4+15 t^3+11 t^2+5 t+1)\equiv t^{5}+1\not\equiv 1\pmod{2}$. Moreover, some $(p,q,u)$ give trivial families of solutions. A careful analysis reveals only 18 interesting triples. They are the following
\begin{align*}
\{&(16, -24, 13), (16, -12, 1), (32, -48, 29), (32, -24, 5), (64, -96, 61), (64, -48, 13),\\
  &(64, -24, 1), (128, -192, 125), (128, -96, 29), (128, -48, 5), (256, -384, 253), \\
  &(256, -192, 61), (256, -96, 13), (256, -48, 1), (512, -768, 509), (512, -384, 125), \\
  &(512, -192, 29), (512, -96, 5)\}.
\end{align*}
A quick look at the above triples motivates the conjecture which says that for each $k\in\N_{+}$, the triple $p=2^{k}, q=-3\cdot 2^{k-1}, u=2^{k}-3$ leads to a non-trivial family of solutions of (\ref{maincong}). Thus, for a given $k\in\N_{\geq 2}$ we consider the corresponding sequence
$$
p_{k,n}=2^{2n+k}-3\cdot 2^{n+k-1}+2^{k}-3
$$
and prove the following.

\begin{thm}
For each $k\in\N_{\geq 2}$ the sequence $(p_{k,n})_{n\in\N_{+}}$, where $p_{k,n}=2^{2n+k}-3\cdot2^{n+k-1}+2^{k}-3$, is a non-trivial sequence of solutions of the congruence {\rm (\ref{maincong})} with $m=2$ and $r=0$.
\end{thm}
\begin{proof}
We prove that for each $k\in\N_{\geq 2}$ and $n\in\N_{+}$, the number $p_{k,n}=2^{2n+k}-3\cdot 2^{n+k-1}+2^{k}-3$ is a solution of the congruence (\ref{maincong}). Applying the first equality from Lemma \ref{lem1} with $a=n+k-1, m=2^{n+1}-3, r=2^{k}-3$, and using the first two identities from Lemma \ref{lem2}, we get the following
\begin{align*}
B_{p_{n}}(t)&=B_{2^{n+k-1}-2^{k}+3}(t)B_{2^{n+1}-3}(t)+B_{2^{k}-3}B_{2^{n+1}-2}(t).
\end{align*}
Next, applying Lemma \ref{lem1} with the number $2^{n+k-1}-2^{k}+3=2^{k}(2^{n-1})-(2^{k}-3)$ we get the equality
\begin{align*}
B_{p_{k,n}}(t)&=(B_{3}(t)B_{2^{n-1}}(t)+B_{2^{k}-3}(t)B_{2^{n-1}-1}(t))B_{2^{n+1}-3}(t)+B_{2^{k}-3}(t)B_{2^{n+1}-2}(t)\\
              &=\left((t+1)t^{n-1}+\left(t\frac{t^{k-2}-1}{t-1}+\frac{t^{k-1}-1}{t-1}\right)\frac{t^{n-1}-1}{t-1}\right)\left(t\frac{t^{n-1}-1}{t-1}+\frac{t^n-1}{t-1}\right)\\
              &\quad+t\left(t\frac{t^{k-2}-1}{t-1}+\frac{t^{k-1}-1}{t-1}\right)\frac{t^{n}-1}{t-1}\\
              &=\left((t+1)t^{n-1}+\frac{(2t^{k-1}-t-1)(t^{n-1}-1)}{(t-1)^2}\right)\frac{2t^{n}-t-1}{t-1}+\frac{t(2t^{k-1}-t-1)(t^{n}-1)}{(t-1)^{2}}.
\end{align*}
In order to finish the proof it is enough to show that $B_{p_{k,n}}(t)\equiv 1\pmod{2}$ and that the sequence $(p_{k,n})_{n\in\N_{+}}$ is non-trivial. Working modulo 2 we have
\begin{align*}
B_{p_{k,n}}(t)&\equiv (t+1)t^{n-1}+\frac{t^{n-1}+1}{t+1}+t\frac{t^{n}+1}{t+1}\\
            &\equiv \frac{(t^2+1)t^{n-1}+t^{n-1}+1+t^{n+1}+t}{t+1}\\
            &\equiv \frac{2t^{n+1}+2t^{n-1}+t+1}{t+1}\equiv 1\pmod{2}.
\end{align*}

Finally, we prove that the sequence $(p_{k,n})_{n\in\N_{+}}$ is non-trivial. In order to shorten the notation let us put $V_{k,n}(t)=B_{p_{k,n}}(t)$. First of all, let us note that by a simple induction argument we can write the expanded form of the polynomials $V_{2,n}(t), V_{3,n}(t)$ in the following way: $V_{2,1}(t)=1+2t, V_{3,1}(t)=1+2t+2t^2$ and for $n\geq 2$ we have
\begin{align*}
V_{2,n}(t)&=1+2\sum_{i=1}^{n-1}(i+1)t^{i}+2t^{n-1}\sum_{i=1}^{n-1}(n+1-i)t^{i}+2t^{2n-1},\\
V_{3,n}(t)&=1+6t+2\sum_{i=2}^{n-1}(3i+1)t^{i}+2(3n-1)t^{n}+2(3n-4)t^{n+1}+2t^{n}\sum_{n=2}^{n-1}(3n-3i-2)t^{i}.\\
\end{align*}
In particular, the set of all coefficients of the sequence of polynomials $(B_{p_{2,n}}(t))_{n\in\N}$, is equal to the set of even positive integers and the number 1. In case of the sequence of polynomials $(B_{p_{3,n}}(t))_{n\in\N}$, the set of coefficients contains all positive integers  $\equiv 2, 4\pmod{6}$ and the numbers 1 and 6.

Next, a simple calculation using the explicit form of the polynomial $V_{k,n}(t)$ leads to the identity
\begin{equation}\label{Vformula}
V_{k+1,n}(t)=(t+1)V_{k,n}(t)-tV_{k-1,n}(t).
\end{equation}
Let us note that
$$
e_{k,n}:=\op{deg}V_{k,n}(t)=\op{max}\{2n-1,2n+k-5,n+k-2\}.
$$
In particular, if $k=3$ and $n\in\N_{\geq 2}$ then $e_{3,n}=2n-1$. If $k\geq 4$ and $n\in\N_{\geq 3}$ then we have $e_{k,n}=2n+k-5$. Let us write
$$
V_{k,n}(t)=\sum_{i=0}^{e_{k,n}}c_{i,k,n}x^{i}.
$$
In particular, $c_{0,k,n}=1$ and by comparing coefficients on both sides of the identity (\ref{Vformula}) we get the recurrence relation
\begin{equation}\label{crec}
c_{i,k,n}= c_{i-1,k-1,n}+c_{i,k-1,n}-c_{i-1,k-2,n}.
\end{equation}
The above recurrence does not depend on $n$, so in the sequel, in order to simplify the notation a bit, we write $c_{i,k}$ instead of $c_{i,k,n}$. Next, let us note that the coefficients $c_{i,2}, c_{i,3}$ are known (from the explicit expressions for $V_{2,n}, V_{3,n}$). Thus, we see that $c_{i,k,n}$ can be expressed in terms of these coefficients only. More precisely, by iterating (\ref{crec}) with respect to the first index, we get the chain of equalities
\begin{align*}
&c_{i,k}=c_{i-1,k-1}+c_{i,k-1}-c_{i-1,k-2}\\
       &\;=c_{i-2,k-2}+c_{i-1,k-2}-c_{i-2,k-3}+c_{i,k-1}-c_{i-1,k-2}=c_{i-2,k-2}+c_{i,k-1}-c_{i-2,k-3}\\
       &\;=c_{i-3,k-3}+c_{i-2,k-3}-c_{i-3,k-4}+c_{i,k-1}-c_{i-2,k-3}=c_{i-3,k-3}+c_{i,k-1}-c_{i-3,k-4}\\
       &\;\quad \vdots\\
       &\;=c_{i-(k-3),k-(k-3)}+c_{i,k-1}-c_{i-(k-3),k-(k-2)}=c_{i,k-1}+c_{i-(k-3),3}-c_{i-(k-3),2}.
\end{align*}
Now, by iterating the equality $c_{i,k}=c_{i,k-1}+c_{i-(k-3),3}-c_{i-(k-3),2}$  with respect to the second index, we get the expression for $c_{i,k}$ in the following form
$$
c_{i,k}=c_{i,2}+\sum_{j=0}^{k-3}(c_{i+j+3-k,3}-c_{i+j+3-k,2}).
$$
From the explicit form of coefficients $c_{i,2}=c_{i,2,n}$ and $c_{i,3}=c_{i,3,n}$ we get
\begin{equation*}
c_{j,3,n}-c_{j,2,n}=\begin{cases}\begin{array}{lll}
                                   0         &  & j=0, \\
                                   2         &  & j=1, \\
                                   4j        &  & j\in\{2,\ldots,n-1\}, \\
                                   4n-2      &  & j=n, \\
                                   4n-6      &  & j=n+1, \\
                                   4(2n-j-1) &  & j\in\{n+2,\ldots,2n-2\}, \\
                                   0         &  & j=2n-1.
                                 \end{array}
\end{cases}
\end{equation*}
In particular, we see that with $n\rightarrow+\infty$ the difference $c_{j,3,n}-c_{j,2,n}$ goes to infinity with $j$. Consequently, for any given $k\in\N_{\geq 4}$, the set
$$
\{c_{i,k,n}:\;i\in\{0,\ldots,e_{k,n}\},\;n\in\N_{+}\}
$$
is infinite.
\end{proof}

We note that the family $(p_{k,n})_{n\in\N_{+}}, k\in\N_{\geq 2}$, allows us to give a non-trivial lower bound for the function $\Pi_{0,2}(x)$. More precisely, we prove the following.

\begin{thm}
For $x\gg 0$ we have the following inequality
$$
\Pi_{0,2}(x)\geq \frac{1}{2}\lfloor\log_{2}x\rfloor ^2-\frac{3}{2}\lfloor\log_{2}x\rfloor+2.
$$
\end{thm}
\begin{proof}
For a given $x\gg 0$ we want to estimate the number of solutions of the double inequality $0<p_{k,n}<x$ in positive integers $k, n$. Let us observe that if $x>2^{k}$, i.e., $k\leq \log_{2}x$, then the inequality $0<p_{k,n}<x$ is equivalent with the following one
$$
n\leq  \log_{2}\left(\frac{1}{4}\left(3+\sqrt{\frac{16x+48-7\cdot 2^{k}}{2^{k}}}\right)\right).
$$
Moreover, we note the important property:
\begin{equation}\label{property}
p_{k_{1},n_{1}}=p_{k_{2},n_{2}}\Longleftrightarrow k_{1}=k_{2}\;\mbox{and}\;n_{1}=n_{2}.
\end{equation}
Indeed, let us suppose that for some $k_{1}, k_{2}, n_{1}, n_{2}\in\N_{+}$ we have the equality $p_{k_{1},n_{1}}=p_{k_{2},n_{2}}$ . Without loss of generality we can assume that $k_{1}\leq k_{2}$. The equality $p_{k_{1},n_{1}}=p_{k_{2},n_{2}}$, after division by $2^{k_{1}}$, gives
$$
2^{2n_{1}}-3\cdot 2^{n_{1}-1}+1=2^{2n_{2}+k_{2}-k_{1}}-3\cdot 2^{n_{2}+k_{2}-k_{1}-1}+2^{k_{2}-k_{1}}.
$$
If $n_{1}=1$, we get the equality $2=2^{k_{2}-k_{1}}(2^{2n_{2}}-3\cdot 2^{n_{2}-1}+1)$. If $n_{2}=1$, then we get $k_{1}=k_{2}$. If $n_{2}>1$, then $2^{2n_{2}}-3\cdot 2^{n_{2}-1}+1>2$ and we get a contradiction. If $n_{1}>1$ and $k_{1}<k_{2}$, then again we get a contradiction. Thus, if $n_{1}>1$, then necessarily $k_{1}=k_{2}$. Consequently $n_{1}=n_{2}$ and the equivalence (\ref{property}) is true.

The above property implies that for any two different pairs $(k_{1}, n_{1}), (k_{2},n_{2})$ we get different numbers $p_{k_{1},n_{1}}, p_{k_{2},n_{2}}$. We thus get that
$$
\Pi_{0,2}(x)\geq \sum_{k=2}^{\lfloor\log_{2}x\rfloor}\log_{2}\left(\frac{1}{4}\left(3+\sqrt{\frac{16x+48-7\cdot 2^{k}}{2^{k}}}\right)\right).
$$

The sum in the above inequality is quite complicated and we simplify it a bit. We note the following inequality
$$
\frac{1}{4}\left(3+\sqrt{\frac{16x+48-7\cdot 2^{k}}{2^{k}}}\right)>\frac{1}{4}\sqrt{\frac{x}{2^{k-3}}}=\sqrt{\frac{x}{2^{k+1}}}
$$
which is simple consequence of the assumption $x>2^{k}$. We thus get that
\begin{align*}
\Pi_{0,2}(x)&\geq\sum_{k=2}^{\lfloor\log_{2}x\rfloor}\frac{1}{2}\log_{2}\frac{x}{2^{k+1}}=\frac{1}{2}\log_{2}\left(\prod_{k=2}^{\lfloor\log_{2}x\rfloor}\frac{x}{2^{k+1}}\right)\\
            &=\frac{1}{2}\log_{2}\left(\frac{x^{\lfloor\log_{2}x\rfloor}}{2^{\frac{1}{2}(\lfloor\log_{2}x\rfloor ^2+3\lfloor\log_{2}x\rfloor -4)}}\right)\\
            &=\frac{1}{2}(\log_{2}x \lfloor\log_{2}x\rfloor-\frac{1}{2}(\lfloor\log_{2}x\rfloor ^2+3\lfloor\log_{2}x\rfloor -4))\\
            &\geq \frac{1}{2}\lfloor\log_{2}x\rfloor ^2-\frac{3}{2}\lfloor\log_{2}x\rfloor+2,
\end{align*}
and we are done.
\end{proof}

\begin{rem}
{\rm The bound obtained in the above theorem seems to be far from optimal. Indeed, there are exactly 145 numbers of the form $p_{k,n}$ which are less then $2^{26}$. This is less then $6\%$ of the solutions of (\ref{maincong}) found by computer search. }
\end{rem}

We performed similar analysis for the case $(r,m)=(1,2)$. Unexpectedly, we were unable to find a double family of solutions. More precisely, we started with the set $A_{1,2}$ containing exactly 134 solutions of (\ref{maincong}) satisfying $n\leq 10^5$. The set $B_{1,2}$ containing all quadruplets $(a,b,c,d)$ such that $a, b, c, d\in A_{1,2}$ and $a<b<c<d$ has 12840751 elements. The corresponding set $C_{1,2}$ of triplets $(p,q,u)$ contains 140 elements. A careful analysis of the set $C_{1,2}$ reveals only eight interesting triplets and this number can be further reduced to four (due to the fact that some triplets lead to sequences which are subsequences of the other). They are the following
\begin{equation*}
\{(64, -36, -1), (128, -72, -5), (1024, -408, -1), (1024, -288, -13)\}.
\end{equation*}
In this way we get four non-trivial infinite sequences of solutions of (\ref{maincong}). The triplets correspond to the sequences $(s_{i,n})_{n\in\N}, i=0,1,2,3$, where
\begin{equation*}
\begin{array}{lll}
  s_{0,n}=2^{2n+4}-9\cdot 2^{n+1}-1,  &  & s_{1,n}=2^{2n+5}-9\cdot 2^{n+2}-5, \\
  s_{2,n}=2^{2n+8}-9\cdot 2^{n+4}-13, &  & s_{3,n}=2^{2n+8}-51\cdot 2^{n+2}-1.
\end{array}
\end{equation*}

We confirm our findings in the following.

\begin{thm}
There is an infinite non-trivial sequence of solutions of the congruence {\rm (\ref{maincong})} with $m=2$ and $r=1$.
\end{thm}
\begin{proof}
Because in each case the proof goes exactly in the same way we prove the non-triviality only for the sequence $s_{0,n}=2^{2n+4}-9\cdot 2^{n+1}-1$. We write $s_{0,n}=2^{n+1}(2^{n+3}-9)-1$ and apply the second identity from Lemma \ref{lem1} with $a=n+1, m=2^{n+3}-9, r=1$. Next, using suitable identities from Lemma \ref{lem2} we get
\begin{align*}
B_{s_{0,n}}(t)&=B_{2^{n+1}-1}(t)B_{2^{n+3}-9}(t)+B_{1}(t)B_{2^{n+3}-10}(t)\\
            &=B_{2^{n+1}-1}(t)B_{2^{n+3}-9}(t)+tB_{2^{n+2}-5}(t)\\
            &=\frac{t^{n+1}-1}{t-1}\left(t\frac{t^{n-1}-1}{t-1}+(t^2+t+1)\frac{t^{n}-1}{t-1}\right)+t\left(t\frac{t^{n-1}-1}{t-1}+(t+1)\frac{t^{n}-1}{t-1}\right)\\
            &=\frac{t^{2n+3}+t^{2n+2}+2t^{2n+1}-2t^{n+2}-4t^{n+1}-2t^{n}-2t^{3}+2t^{2}+3t+1}{(t-1)^{2}}.
\end{align*}
Let us observe that $\op{deg}B_{s_{0,n}}(t)=2n+1$ and thus
\begin{align*}
B_{s_{0,n}}(t)-1-t\frac{t^{2n+1}-1}{t-1}\equiv \frac{2t(t+1)(t^{2n}-t^{n}-t^{n-1}-t+2)}{(t-1)^{2}}\equiv 0\pmod{2}.
\end{align*}
In order to show that the sequence $(s_{0,n})_{n\in\N}$ is non-trivial it is enough to note the following identity
\begin{equation}\label{pqconnection}
B_{s_{0,n}}(t)=(t+1)B_{p_{2,n}}(t)+t^{2}\frac{t^{2n}-1}{t-1}.
\end{equation}
The proof of this identity is an immediate consequence of the obtained explicit expression of $B_{p_{2,n}}(t)$ and $B_{s_{0,n}}(t)$. Thus, the non-triviality of the sequence $(p_{2,n})_{n\in\N}$ immediately implies the non-triviality of the sequence $(s_{0,n})_{n\in\N}$. An explicit expansion can be now deduced easily. We get $B_{s_{0,1}}(t)=B_{27}(t)=(t+1)^3$ and for $n\geq 2$ we have
$$
B_{s_{0,n}}(t)=1+5t+\sum_{i=2}^{n-1}(4i-3)t^{i}+(4n+1)t^{n}+(4n-1)t^{n+1}+\sum_{i=n+2}^{2n}(4(2n-i)+3)t^{i}+t^{2n+1}.
$$
In particular, the set of all coefficients of the sequence of polynomials $(B_{s_{0,n}}(t))_{n\in\N}$ is equal to the set of odd positive integers. Our proof is finished.
\end{proof}

Although we tried hard, we were unable to extend the sequences $(s_{i,n})_{n\in\N_{+}}$ for $i=0,1,2,3,$ to an infinite family of sequences. Anyway, the existence of four infinite families of non-trivial solutions (together with trivial families) lead immediately to the following.

\begin{cor}
We have $\Pi_{1,2}(x)\geq \log_{2}x$.
\end{cor}

Our theoretical result and numerical computations suggest the following series of questions.

\begin{ques}
Does the limit $\lim_{x\rightarrow +\infty}\frac{\Pi_{0,2}(x)}{\Pi_{1,2}(x)}$ exists?
\end{ques}

\begin{ques}
What is true order of magnitude of the function $\Pi_{r,2}(x)$ for $r\in\{0,1\}$?
\end{ques}

The graphs of the quotients $\Pi_{r,2}(x)/(\log_{2}x)^{2}, r=0,1,$ presented below suggest the following
\begin{ques}
Let $r\in\{0,1\}$. Is it true that $\lim_{x\rightarrow+\infty}\frac{\Pi_{r,2}(x)}{(\log_{2}x)^{2}}=+\infty$?
\end{ques}

\begin{figure}[h]\label{Pic3} 
       \centering
         \includegraphics[width=4in]{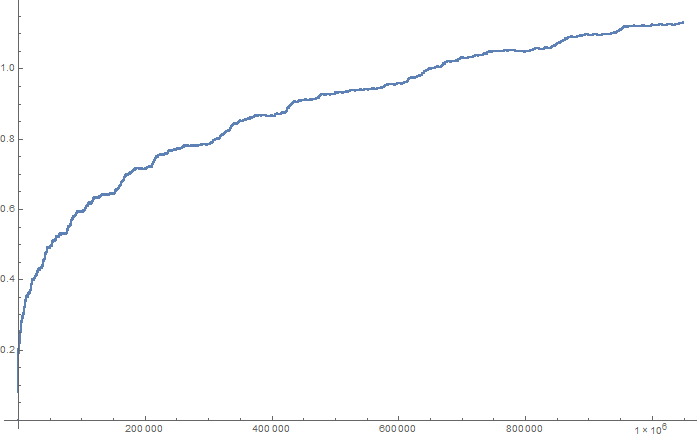}
        \caption{Plot of the function $\protect \Pi_{0,2}(x)/(\log_{2}x)^{2}$ for $\protect 2\leq x\leq 2^{20}$}
       \label{fig:disc}
    \end{figure}

\begin{figure}[h]\label{Pic4} 
       \centering
         \includegraphics[width=4in]{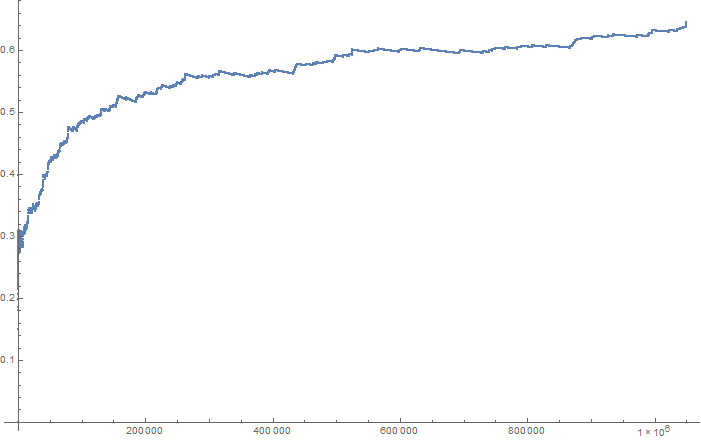}
       \caption{Plot of the function $\protect \Pi_{1,2}(x)/(\log_{2}x)^{2}$ for $\protect 2\leq x\leq 2^{20}$}
       \label{fig:disc}
    \end{figure}

\section{The case of $m=3$}\label{sec3}

In this section we investigate the congruence (\ref{maincong}) in the case $m=3$. Similarly, as in the case of $m=2$ we performed numerical calculations in order to see how many solutions we can expect. As we already noted, for each $m$ and $r=1, 2$ we have solutions related to trivial families $(2^{n+1}-1)_{n\in\N}$ (for $r=1$) and $(2^{n+2}-3)_{n\in\N}$ (for $r=2$). Thus, we were interested in solutions which are not of the form $2^{n+1}-1, 2^{n+2}-3$. Our initial search for solutions of (\ref{maincong}) was for $n\leq 2^{26}$. However, we immediately realized that there are very few interesting solutions in this range and extend the search for $n\leq 2^{30}$. The script needed to compute the solutions was written in the Mathematica computational package \cite{Wol}. The computations in this range took about one week on the personal laptop of the author. Next, our colleague Maciej Gawron \cite{Gaw1}, wrote a very fast program in C++ in order to extend the search up to $2^{40}$ (the program took about one week on the personal laptop of the author). We list the solutions $\leq 2^{40}$ with their binary representation in the tables below.

\begin{equation*}
\begin{array}{|l|l|}
\hline
 n & \mbox{binary expansion of}\;n \\
 \hline
 {\bf 19} & {\bf (1 0 0 1 1)_{2}} \\
 181 & (1 0 1 1 0 1 0 1)_{2} \\
 29899 & (1 1 1 0 1 0 0 1 1 0 0 1 0 1 1)_{2} \\
 40123 & (1 0 0 1 1 1 0 0 1 0 1 1 1 0 1 1)_{2} \\
 44659 & (1 0 1 0 1 1 1 0 0 1 1 1 0 0 1 1)_{2} \\
 72361 & (1 0 0 0 1 1 0 1 0 1 0 1 0 1 0 0 1)_{2} \\
 {\bf 87211} & {\bf (1 0 1 0 1 0 1 0 0 1 0 1 0 1 0 1 1)_{2} }\\
 183439 & (1 0 1 1 0 0 1 1 0 0 1 0 0 0 1 1 1 1)_{2} \\
 373465 & (1 0 1 1 0 1 1 0 0 1 0 1 1 0 1 1 0 0 1)_{2} \\
 2965429 & (1 0 1 1 0 1 0 0 1 1 1 1 1 1 1 0 1 1 0 1 0 1)_{2} \\
 5073589 & (1 0 0 1 1 0 1 0 1 1 0 1 0 1 0 1 0 1 1 0 1 0 1)_{2} \\
 17484211 & (1 0 0 0 0 1 0 1 0 1 1 0 0 1 0 0 1 1 0 1 1 0 0 1 1)_{2} \\
 44733781 & (1 0 1 0 1 0 1 0 1 0 1 0 0 1 0 1 0 1 0 1 0 1 0 1 0 1)_{2} \\
 165459277 & (1 0 0 1 1 1 0 1 1 1 0 0 1 0 1 1 0 1 0 1 0 1 0 0 1 1 0 1)_{2} \\
 1381288843 & (1 0 1 0 0 1 0 0 1 0 1 0 1 0 0 1 1 0 0 1 0 1 1 1 0 0 0 1 0 1 1)_{2} \\
 2572135705 & (1 0 0 1 1 0 0 1 0 1 0 0 1 1 1 1 1 0 1 0 1 1 0 1 0 0 0 1 1 0 0 1)_{2} \\
 2833893901 & (1 0 1 0 1 0 0 0 1 1 1 0 1 0 0 1 1 1 0 0 1 0 1 0 0 0 0 0 1 1 0 1)_{2} \\
 \hline
\end{array}
\end{equation*}
\begin{center}
Table 2. Solutions of the congruence (\ref{maincong}) for the pair $(r,m)=(0,3)$ and $n\leq 2^{40}$.
\end{center}

 \begin{equation*}
 \begin{array}{|l|l|}
 \hline
 n & \mbox{binary expansion of}\;n \\
 \hline
 157 & (1 0 0 1 1 1 0 1)_{2} \\
 4789 & (1 0 0 1 0 1 0 1 1 0 1 0 1)_{2} \\
 12615 & (1 1 0 0 0 1 0 1 0 0 0 1 1 1)_{2} \\
 46257 & (1 0 1 1 0 1 0 0 1 0 1 1 0 0 0 1)_{2} \\
 78765 & (1 0 0 1 1 0 0 1 1 1 0 1 0 1 1 0 1)_{2} \\
 120147 & (1 1 1 0 1 0 1 0 1 0 1 0 1 0 0 1 1)_{2} \\
 201069 & (1 1 0 0 0 1 0 0 0 1 0 1 1 0 1 1 0 1)_{2} \\
 46343011 & (1 0 1 1 0 0 0 0 1 1 0 0 1 0 0 0 1 1 0 1 1 0 0 0 1 1)_{2} \\
 156666811 & (1 0 0 1 0 1 0 1 0 1 1 0 1 0 0 0 1 0 1 1 1 0 1 1 1 0 1 1)_{2} \\
 1235649115 & (1 0 0 1 0 0 1 1 0 1 0 0 1 1 0 1 0 0 0 0 0 1 0 0 1 0 1 1 0 1 1)_{2} \\
 45728246203 & (1 0 1 0 1 0 1 0 0 1 0 1 1 0 0 1 1 1 0 1 1 0 1 0 1 0 0 1 1 0 1 1 1 0 1 1)_{2} \\
 78080354869 & (1 0 0 1 0 0 0 1 0 1 1 0 1 1 1 1 1 0 0 1 1 1 0 1 0 1 0 1 0 0 0 1 1 0 1 0 1)_{2} \\
 95882561515 & (1 0 1 1 0 0 1 0 1 0 0 1 1 0 0 0 0 1 0 1 1 1 1 0 0 0 1 1 1 1 1 1 0 1 0 1 1)_{2} \\
 775752845083 & (1 0 1 1 0 1 0 0 1 0 0 1 1 1 1 0 0 1 1 1 1 0 0 1 0 0 0 0 0 1 1 1 0 0 0 1 1 0 1 1)_{2} \\
 \hline
\end{array}
 \end{equation*}
\begin{center}
Table 3. Solutions of the congruence (\ref{maincong}) for the pair $(r,m)=(1,3)$ and $n\leq 2^{40}$.
\end{center}

\begin{equation*}
\begin{array}{|l|l|}
\hline
 n & \mbox{binary expansion of}\;n \\
 \hline
 83 & (1 0 1 0 0 1 1)_{2} \\
 359 & (1 0 1 1 0 0 1 1 1)_{2} \\
 631 & (1 0 0 1 1 1 0 1 1 1)_{2} \\
 2633 & (1 0 1 0 0 1 0 0 1 0 0 1)_{2} \\
 37579 & (1 0 0 1 0 0 1 0 1 1 0 0 1 0 1 1)_{2} \\
 43411 & (1 0 1 0 1 0 0 1 1 0 0 1 0 0 1 1)_{2} \\
 52409 & (1 1 0 0 1 1 0 0 1 0 1 1 1 0 0 1)_{2} \\
 80723 & (1 0 0 1 1 1 0 1 1 0 1 0 1 0 0 1 1)_{2} \\
 374383 & (1 0 1 1 0 1 1 0 1 1 0 0 1 1 0 1 1 1 1)_{2} \\
 10717481 & (1 0 1 0 0 0 1 1 1 0 0 0 1 0 0 1 0 0 1 0 1 0 0 1)_{2} \\
 23629421 & (1 0 1 1 0 1 0 0 0 1 0 0 0 1 1 1 0 0 1 1 0 1 1 0 1)_{2} \\
 26528431 & (1 1 0 0 1 0 1 0 0 1 1 0 0 1 0 1 0 1 0 1 0 1 1 1 1)_{2} \\
 44767195 & (1 0 1 0 1 0 1 0 1 1 0 0 0 1 0 1 1 1 1 1 0 1 1 0 1 1)_{2} \\
 \hline
\end{array}
\end{equation*}
\begin{center}
Table 4. Solutions of the congruence (\ref{maincong}) for the pair $(r,m)=(2,3)$ and $n\leq 2^{40}$.
\end{center}

In Table 2 we wrote the solutions 19, 87211 of (\ref{maincong}) in bold. The reason is that we were able to extend them to infinite sequence and prove Theorem  \ref{30case} below. Indeed, let us note that
$$
19=2^{4}+2^{1}+2^{0},\quad 87211=2^{16} + 2^{14} + 2^{12} + 2^{10} + 2^{7} + 2^{5} + 2^{3} + 2^{1} + 2^{0}.
$$
This lead us to conjecture that maybe the sequence $(h_{n})_{n\in\N}$, where
$$
h_{n}=\sum_{i=1}^{n}2^{2i-1}+\sum_{i=n+1}^{2n}2^{2i}+1=\frac{2}{3}(2^{2n}-1)(2^{2n+1}+1)+1,
$$
contains more solutions of the congruence (\ref{maincong}). We have $h_{1}=19, h_{4}=87211$. As we will see our expectation is true. However, before we state the result we need to introduce the related sequence $(\alpha_{n})_{n\in\N}$. Here
$$
\alpha_{n}=\frac{1}{3}(2^{n}-(-1)^{n}),
$$
and the number $\alpha_{n}$ is called the $n$th Jacobstahl number \cite{Hor}. It has an amusing property which says that the maximum values of the Stern sequence in the interval $[0,2^{n-1}]$ is attained  at $s_{\alpha_{n}}$ (and also at $\beta_{n}=\frac{1}{3}(5\cdot 2^{n-2}+(-1)^{n})$). Moreover, $s_{\alpha_{n}}=F_{n}$ for $n\in\N_{+}$, where $F_{n}$ is the $n$th Fibonacci number.

It is known that the sequence $(B_{\alpha_{n}}(t))_{n\in\N}$ satisfies the following recurrence relation
$$
B_{\alpha_{0}}(t)=0,\quad B_{\alpha_{1}}(t)=1,\quad B_{\alpha_{n}}(t)=B_{\alpha_{n-1}}(t)+tB_{\alpha_{n-2}}(t).
$$
This is \cite[Lemma 2.2]{DKT}. One can easily check that the subsequence $(B_{\alpha_{2n}}(t))_{n\in\N}$ satisfies the recurrence relation
$$
B_{\alpha_{2n}}(t)=(2t+1)B_{\alpha_{2(n-1)}}(t)-t^{2}B_{\alpha_{2(n-2)}}(t).
$$

\begin{thm}\label{30case}
There are infinitely many $n\in\N$ such that $B_{n}(t)\equiv 1\pmod{3t(t+1)}$. More precisely, let $h_{n}=\frac{2}{3}(2^{2n}-1)(2^{2n+1}+1)+1$. Then
\begin{equation*}
B_{h_{\frac{3^{n}-1}{2}}}(t)\equiv 1\pmod{3t(t+1)}.
\end{equation*}
\end{thm}
\begin{proof}
We write $W_{n}(t)=B_{h_{n}}(t)\in\Z[t]$. It is clear that the only non-trivial thing to prove is the vanishing modulo 3 of the difference $B_{h_{\frac{3^{n}-1}{2}}}(t)-1$. However, before we prove this we need to express the polynomial $B_{h_{n}}(t)$ in a more convenient form. We write
$h_{n}=2^{2n+2}\alpha_{2n}+2\alpha_{2n}+1$ and apply the first identity from Lemma \ref{lem1} with $a=2n+2, m=\alpha_{2n}$ and $r=2\alpha_{2n}+1$. We thus get
\begin{align*}
B_{h_{n}}(t)&=B_{2^{2n+2}-2\alpha_{2n}-1}(t)B_{\alpha_{2n}}(t)+B_{2\alpha_{2n}+1}(t)B_{\alpha_{2n}+1}(t)\\
            &=B_{2(2^{2n+1}-\alpha_{2n}-1)+1}(t)B_{\alpha_{2n}}(t)+(B_{\alpha_{2n}}(t)+B_{\alpha_{2n}+1}(t))B_{\alpha_{2n}+1}(t)\\
            &=(B_{2\cdot2^{2n}-(\alpha_{2n}+1)}(t)+B_{2^{2n+1}-\alpha_{2n}}(t))B_{\alpha_{2n}}+(B_{\alpha_{2n}}(t)+B_{\alpha_{2n}+1}(t))B_{\alpha_{2n}+1}(t)\\
            &=(B_{2^{2n}-\alpha_{2n}-1}(t)B_{2}(t)+B_{\alpha_{2n}+1}(t)B_{1}(t)+B_{2^{2n}-\alpha_{2n}}(t)B_{2}(t)\\
            &\quad +B_{\alpha_{2n}+1}(t)B_{1}(t))B_{\alpha_{2n}}(t)+(B_{\alpha_{2n}}(t)+B_{\alpha_{2n}+1}(t))B_{\alpha_{2n}+1}(t)\\
            &=(tB_{2\alpha_{2n}}(t)+B_{\alpha_{2n}+1}(t)+tB_{2\alpha_{2n}+1}(t)+B_{\alpha_{2n}}(t))B_{\alpha_{2n}}(t)+(B_{\alpha_{2n}}(t)\\
            &\quad +B_{\alpha_{2n}+1}(t))B_{\alpha_{2n}+1}(t)\\
            &=(t^2+t+1)B_{\alpha_{2n}}(t)^2+(t+2)B_{\alpha_{2n}}(t)B_{\alpha_{2n}+1}(t)+B_{\alpha_{2n}+1}(t)^2.\\
\end{align*}
Now let us observe that the sequence $(\alpha_{2n})_{n\in\N}$ satisfies the recurrence relation $\alpha_{2(n+1)}=4\alpha_{2n}+1$. Thus
\begin{equation*}
B_{\alpha_{2(n+1)}}(t)=B_{4\alpha_{2n}+1}(t)=(t+1)B_{\alpha_{2n}}(t)+B_{\alpha_{2n}+1}(t)
\end{equation*}
and we can express the polynomial $B_{\alpha_{2n}+1}(t)$ in the following form
\begin{equation*}
B_{\alpha_{2n}+1}(t)=B_{\alpha_{2(n+1)}}(t)-(t+1)B_{\alpha_{2n}}(t).
\end{equation*}
Finally, we write $B_{h_{n}}(t)=F(B_{\alpha_{2n}}(t), B_{\alpha_{2(n+1)}}(t))$, where
$$
F(X,Y)=(t^2+t+1)X^2-(t+1)XY+Y^2.
$$
Now let us recall that the sequence $(B_{\alpha_{2n}}(t))_{n\in\N}$ satisfies a linear difference equation of degree 2. Consequently, the sequence ($B_{h_{n}}(t))_{n\in\N}$ needs to satisfy a linear difference equation of degree bounded by 4. It is an easy task for the computer to find the corresponding recurrence. Indeed, let us write $W_{n}(t)=B_{h_{n}}(t)$ and observe that the sequence $(W_{n}(t))_{n\in\N}$ satisfies the following recurrence relation:
\begin{equation*}
\begin{cases}
\begin{array}{lll}
  W_{0}(t) & = & 1, \\
  W_{1}(t) & = & 3t^2+t+1, \\
  W_{2}(t) & = & 7t^4+17t^3+7t^2+7t+1, \\
  W_{n}(t) & = & (3t^2+4t+1)W_{n-1}(t)-t^2(3t^2+4t+1)W_{n-2}(t)+t^6W_{n-3}(t).
\end{array}
\end{cases}
\end{equation*}
The sequence $(W_{n}(t))_{n\in\N}$ is in the null space of the linear difference operator $\Theta_{t}$, where
\begin{equation*}\label{theta}
\Theta_{t}=(T-t^2)(T^2-(2t^2+4t+1)T+t^4),
\end{equation*}
where $T$ is the (forward) shift operator, i.e., $T((a_{n})_{n\in\N})=(a_{n+1})_{n\in\N}$. Let us observe that the discriminant of the quadratic factor of $\Theta_{t}$ is $(1 + 2 t)^2 (1 + 4 t)$. Thus, the idea is to introduce a new variable $u$ connected with $t$ by the equality $t=\phi(u)=\frac{1}{4}(u^2-1)$. Let us observe that the reduction modulo 3 does not change the degree of $\phi$ (here we understood $\phi$ as a polynomial in the variable $u$), and we obtain the equivalence
$$
W_{n}(t)\equiv 1\pmod{3}\;\Longleftrightarrow\;W_{n}(\phi(u))\equiv 1\pmod{3}.
$$
With the substitution $t=\phi(u)$ we obtain the factored form of the corresponding difference operator
$$
\Theta_{\phi(u)}=\left(T-\frac{1}{16}(u-1)^4\right)\left(T-\frac{1}{16}(u+1)^4\right)\left(T-\frac{1}{16}\left(u^2-1\right)^2\right).
$$
If we write $w_{n}(u)=W_{n}\left(\frac{1}{4}(u^2-1)\right)$, then the operator $\Theta_{\phi(u)}$ anihilates the sequence $(w_{n}(u))_{n\in\N}$ and splits into linear factors. Thus, using standard methods of solving linear difference equations with constant coefficients, we obtain the expression for $w_{n}(u)$ in the following (closed) form:
$$
w_{n}(u)=\frac{(u-1)^2(u^2+3)}{16u^2}\left(\frac{u-1}{2}\right)^{4n}+\frac{(u+1)^2(u^2+3)}{16u^2}\left(\frac{u+1}{2}\right)^{4n}-\frac{(u^2-3)(u^2-1)}{8u^2}\left(\frac{u^2-1}{4}\right)^{2n}.
$$
Moreover, reducing modulo 3 we immediately get the congruence $w_{n}(u)\equiv (u-1)^{4n+2}+(u+1)^{4n+2}+(u^2-1)^{2n+1}\pmod{3}$. In order to finish the proof of our theorem it is enough to prove that $w_{\frac{3^{n}-1}{2}}(u)\equiv 1\pmod{3}$. Let us recall that $(p+q)^{3^{k}}\equiv p^{3^{k}}+q^{3^{k}}\pmod{3}$ for any $k\in\N$ and variables $p, q$. We thus have
\begin{align*}
w_{\frac{3^{n}-1}{2}}(u)&\equiv (u-1)^{2\cdot 3^{n}}+(u+1)^{2\cdot 3^{n}}+(u^2-1)^{3^{n}}\\
                        &\equiv (u^{3^{n}}-1)^2+(u^{3^{n}}+1)^2+u^{2\cdot 3^{n}}-1\equiv 1\pmod{3}.
\end{align*}

We are left with the problem of showing that the set of coefficients of all polynomials $B_{h_{(3^{n}-1)/2}}(t)$ is infinite. However, this is an immediate consequence of the explicit form of $B_{\alpha_{n}}(t)$ obtained in \cite{DKT}. More precisely, we have
\begin{equation*}
B_{\alpha_{n}}(t)=\sum_{j=0}^{\lfloor\frac{n-1}{2}\rfloor}\binom{n-1-j}{j}t^{j}.
\end{equation*}
In particular the set of coefficients of all $B_{\alpha_{n}}(t), n\in\N$ is infinite and it is clear that the same statement is true for the set of coefficients of polynomials
$$
B_{h_{n}}(t)=(t^2+t+1)B_{\alpha_{2n}}(t)^2+(t+2)B_{\alpha_{2n}}(t)B_{\alpha_{2n}+1}(t)+B_{\alpha_{2n}+1}(t)^2, n\in\N
$$
and thus the set of coefficients of the polynomials $B_{h_{(3^{n}-1)/2}}(t), n\in\N$.
\end{proof}

As an immediate consequence of the existence of the sequence constructed in the proof above we get

\begin{cor}
We have $\Pi_{0,3}(x)\geq \log_{3}\log_{2}\left(\frac{1}{2}(1+\sqrt{3(3+4x)})\right)$.
\end{cor}
\begin{proof}
Just solve the inequality $h_{(3^{n}-1)/2}\leq x$ with respect to $n$.
\end{proof}

\begin{rem}
{\rm The numbers $H_{n}=h_{\frac{3^{n}-1}{2}}, n\in\N$, grow very quickly. The first five elements for $n=0,\ldots, 4$ are as follows
$$
1, 19, 87211, 6004799458421419, 1948668849774537224271578971004497616455126919851,\ldots .
$$
The growth of the sequence $H_{n}$ is double exponential and a question arises whether it is possible to find an increasing sequence, say $(u_{n})_{\in\N}$, of positive integers with exponential growth and such that for each $n\in\N$ the number $u_{n}$ is a solution of (\ref{maincong}) for $(r,m)=(0,3)$ and it is non-trivial.
}
\end{rem}

\section{Final remarks}\label{sec4}

We finish the paper with report of related computations and formulate some questions and conjectures.

We performed numerical experiments in order to find solutions of (\ref{maincong}) for $m\in\{4,\ldots,10\}$ and $r\in\{0,\ldots,m-1\}$ with $n<2^{40}$. In this range we found very few solutions which can be found in the table below.

\begin{equation*}
\begin{array}{|l|l|l|l|}
\hline
  (r,m) & n & \mbox{binary expansion of}\;n & B_{n}(t) \\
  \hline
  (2,m) & 13  & (1101)_{2}       & 2t^2+2t+1\\
  (3,m) & 19  & (10011)_{2}      & 3t^2+3t+1\\
  (4,5) & 205 & (11001101)_{2}   & 5 t^4+10 t^3+9 t^2+4 t+1\\
  (0,5) & 211 & (11010011)_{2}   & 5 t^4+10 t^3+10 t^2+5 t+1 \\
  (2,5) & 331 & (101001011)_{2}  & 7 t^4+17 t^3+17 t^2+7 t+1 \\
  (2,4) & 629 & (1001110101)_{2} & 6 t^5+14 t^4+18 t^3+14 t^2+6 t+1 \\
  (4,5) & 925 & (1110011101)_{2} & 4 t^6+9 t^5+14 t^4+14 t^3+9 t^2+4 t+1\\
  \hline
\end{array}
\end{equation*}
\begin{center}
Table 5.  Solutions of (\ref{maincong}) for some pairs $(r,m)$ with $m\in\{4,\ldots, 10\}, r\in\{0,\ldots, m-1\}$ and $n<2^{40}$.
\end{center}

In the light of our numerical results we formulate the following.

\begin{ques}
Let $m\in\N_{\geq 4}$ and $r\in\{0,\ldots,m-1\}$ be given. Is it true that the set of all solutions of {\rm (\ref{maincong})} is contained in finitely many trivial families?
\end{ques}

\bigskip

Now we turn our attention to polynomials related to the sequences considered in Section \ref{sec2}. More precisely, in Section \ref{sec2}, for $k\in\N_{\geq 2}$, we considered the sequences $(p_{k,n})_{n\in\N_{+}}$, where
$$
p_{k,n}=2^{2n+k}-3\cdot2^{n+k-1}+2^{k}-3.
$$

Let us note that for each $n\in\N$ the polynomials $B_{p_{2,n}}(t), B_{p_{3,n}}(t)$ are irreducible. In order to see this it is enough to apply Eisenstein's criterion to the polynomial $F_{i,n}(t)=t^{e(p_{i,n})}B_{p_{i,n}}\left(\frac{1}{t}\right), i=2,3$. Indeed, the leading coefficient of $F_{i,n}$ is equal to 1, the constant coefficient is equal to 2 (and thus is not divisible by 4), and all other coefficients are divisible by 2. Thus, we apply Eistenstein's criterion with $q=2$ and get irreducibility of $F_{i,n}(t)$ and thus irreducibility of the polynomial $B_{p_{i,n}}(t)$ for $i=2,3$. It should be noted that this simple approach does not work for all polynomials $B_{p_{k,n}}(t)$ with $k\in\N_{\geq 2}$ and $n\in\N_{+}$ (however, see the conjecture below).

Our computer experiments with $2\leq k\leq 10$ and $n\leq 10^3$ suggest that the following conjectures are true.
\begin{conj}
Let
$$
N_{k}(n)=|\{t\in\R:\;B_{p_{k,n}}(t)=0\}|.
$$
\begin{enumerate}
\item[(1)] If $k\equiv 0\pmod{2}$, then there is $c_{k}\in\N_{+}$ such that the polynomial $B_{p_{k,n}}(t)$ has exactly one real root for $n\in\N_{\geq c_{k}}$. Moreover, under the same assumptions on $k, n$, the function $y=B_{p_{k,n}}(t)$ is increasing for $t\in\R$. In particular, if $k=1$, then $c_{2}=1$, $c_{2i}=3$ for $i=2,\ldots,16$ (with exactly {\rm 2} real roots of the polynomial $B_{p_{2i,2}})$ , $c_{34}=6,\ldots$.

\item[(2)] If $k\equiv 1\pmod{2}$, then there is a number $C_{k}$ such that for each $n\in\N_{+}$ we have $N_{k}(n)\leq C_{k}$.

\item[(3)] The polynomial $B_{p_{k,n}}(t)$ is reducible if and only if $k\in\N_{\geq 3}$ and $n=k-1$. Then we have the identity
$$
B_{p_{k,k-1}}(t)=\left(1+2t\frac{t^{k-2}-1}{t-1}\right)(B_{2,k-2}(t)+2t^{k-2}(1+t))
$$
and both factors are irreducible in $\Q[t]$.
\end{enumerate}
\end{conj}

Note that Eisenstein's criterion with $q=2$ immediately implies the irreducibility of the two factors of $B_{p_{k-1,k}}(t)$ presented in the conjecture above.

Next, let us note that for each $i\in\{0,1,2,3\}$, the polynomial $B_{s_{i,n}}(t)$ is reducible due to the fact that it is divisible by $t+1$. Indeed, we know that $B_{n}(-1)=0$ if and only if $n\equiv 0\pmod{3}$. It is an easy exercise to show that $s_{i,n}\equiv 0\pmod{3}$ for $i\in\{0,1,2,3\}$ and $n\in\N_{+}$, and hence $B_{s_{i,n}}(-1)=0$. However, one can also note the following identities:
\begin{align*}
B_{s_{1,2n}}(t)  &=(t+1)\left(1+3t\frac{t^{2n}-1}{t-1}-2t^{2n}\right)\left(1+3t\frac{t^{2n-1}-1}{t-1}-2t^{3}\frac{t^{2(n-1)}-1}{t^2-1}+t^{2n}\right),\\
B_{s_{1,2n+1}}(t)&=(t+1)\left(1+2t^2+t\frac{t^{2(n+1)}-1}{t-1}-t^{2n+1}(3t+2)\right)\left(1+2t\frac{t^{2n-1}-1}{t-1}-t^{2}\frac{t^{2(n-1)}-1}{t^2-1}+t^{2n}\right).\\
\end{align*}

\begin{conj}
\begin{enumerate}
\item[(1)] If $i=0,2,3$, then the polynomial $B_{s_{i,n}}(t)$ has exactly one real root for $n\in\N_{\geq 2}$.
\item[(2)] If $i=0,2,3$, then the polynomial $B_{s_{i,n}}(t)/(t+1)$ is irreducible in $\Q[t]$.
\item[(3)] The polynomial $B_{s_{1,n}}(t)$ has exactly three real roots for $n\in\N_{\geq 3}$.
\item[(4)] The factors of the polynomial $B_{s_{1,n}}(t)$ are irreducible (here we mean the factors presented before the statement).
\item[(3)] The function $y=B_{s_{0,n}}(t)$ is increasing for $t\in\R$.
\end{enumerate}
\end{conj}

We finish with the following.

\begin{conj}
Let $h_{n}=\frac{2}{3}(2^{2n}-1)(2^{2n+1}+1)+1$. For each $n\in\N$ the polynomial $B_{h_{n}}(t)$ is completely complex, i.e., the equation $B_{h_{n}}(t)=0$ has no real roots.
\end{conj}

\noindent {\bf Acknowledgments.}
The author expresses his gratitude to Maciej Gawron for his kind help in computations of examples corresponding to $m=3$ and $n>2^{30}$ presented in Section \ref{sec3} and to the referee for a careful reading of the manuscript and valuable suggestions, which improved the quality of the paper.

\vskip 1cm
\noindent Maciej Ulas, Institute of Mathematics of the Polish Academy of Sciences, \'{S}wi\c{e}tego Tomasza 30, 31-014 Krak\'{o}w, Poland; email: {\tt maciej.ulas@uj.edu.pl}
\vskip 0.5cm
and
\vskip 0.5cm
\noindent Jagiellonian University, Faculty of Mathematics and Computer Science, Institute of
Mathematics, {\L}ojasiewicza 6, 30-348 Krak\'ow, Poland.

\end{document}